\documentclass[en,oneside,bbfont,article]{amsart} 
\usepackage[lmargin = 30mm, rmargin = 30mm, bmargin = 25mm, tmargin=25mm]{geometry}
\usepackage[utf8]{inputenc}

\usepackage{mathtools}

\newtheorem{theorem}{Theorem}[section]

\usepackage{mathrsfs}
\usepackage{amssymb}
\usepackage{dsfont}
\usepackage{amsthm}
\usepackage{amsmath}
\usepackage{comment}
\usepackage{caption,subcaption}
\usepackage{scalefnt}
\usepackage{float,graphicx,verbatim}
\usepackage{tikz}
\usetikzlibrary{calc}
\usepackage{pgfplots}
\pgfplotsset{compat=1.16}
\RequirePackage[colorlinks=true,linkcolor=black,
	citecolor=blue,urlcolor=blue]{hyperref}

\usepackage[backend=biber,style=numeric,citestyle=numeric,maxbibnames=99]{biblatex}
\renewbibmacro{in:}{}
\appto{\bibsetup}{\sloppy}

\newcommand\N{\mathbb{N}}

\newcommand\R{\mathbb{R}}
\newcommand\E{\mathds{E}}
\newcommand\p{\mathds{P}}
\newcommand\1{\mathds{1}}
\newcommand\eqd{\overset{d}{=}}

\newcommand\Var{\mathrm{Var}}
\newcommand\co{\mathsf{c}}
\newcommand\nf[1]{\normalfont{#1}}

\newcommand{\D}{\mathrm{d}}
\newcommand{\ov}[1]{\overline{#1}}

\title{A Gaussian approximation
theorem for L\'evy processes}
\date{\today}
\author{David Bang, Jorge Gonz\'alez C\'azares \& Aleksandar Mijatovi\'c}

\address{Department of Statistics, University of Warwick, \and The Alan Turing Institute, UK}

\email{david.bang@warwick.ac.uk}
\email{jorge.gonzalez-cazares@warwick.ac.uk}
\email{a.mijatovic@warwick.ac.uk}

\begin{document}

\begin{abstract}
Without higher moment assumptions,
this note establishes the decay of the Kolmogorov distance in a central limit theorem for L\'evy processes. This theorem can be viewed as a continuous-time extension of the classical random walk result by Friedman, Katz and Koopmans~\cite{FriedmanKatzCLT}. 
\end{abstract}



\maketitle

\section{Introduction}
\label{sec:intro}
The classical central limit theorem (CLT), applied to a one-dimensional L\'evy processes $X=(X_t)_{t\ge0}$ with zero mean and finite variance, states that $\p(X_t/\sqrt{t}\le x)\to\Phi(x/\sigma)$ as $t\to\infty$ for all $x\in\R$, where
$\sigma^2$ is the variance of $X_1$ and
$\Phi$ is the distribution function of a standard normal random variable $Z$. Since the law $\Phi$ has a bounded density, this weak convergence is well-known to imply the convergence in the Kolmogorov distance  $\sup_{x\in\R}|\p(X_t/\sqrt{t}\le x)-\Phi(x/\sigma)|\to 0$ as $t\to\infty$, see e.g.~\cite[1.8.31--32, p.~43]{MR1353441}. 
It is natural to inquire about the rate of this convergence without additional assumptions on the L\'evy process. In this paper we answer this question, thus extending to the continuous-time setting the classical random walk result by Friedman, Katz and Koopmans~\cite{FriedmanKatzCLT}.


\begin{theorem}\label{thm:int-1}
Let $X=(X_t)_{t\geq0}$ be a L\'evy process satisfying 
$\sigma :=\E[X_1^2]^{1/2}\in(0,\infty)$,
$\E X_1=0$ and $X_0=0$ a.s. If the L\'evy measure~$\nu$ of $X$ is nontrivial, choose $\kappa\ge 1$ such that  $0<\nu((-\kappa,\kappa))\leq \infty$ and otherwise set $\kappa\coloneqq 1$. 
Defining  
$\sigma_t^2\coloneqq \sigma^2-\int_{\R\setminus(-\kappa\sqrt{t},\kappa\sqrt{t})}x^2 \nu(\D x)$ for $t>0$,
we have 
\begin{equation}
\label{eq:int-1}
\int_1^\infty 
\sup_{x\in\R}\big|
    \p\big(X_t/\sqrt{t}\le x\big)
    -\Phi\big(x/\sigma_t\big)\big|
        \frac{\D t}{t}<\infty.
\end{equation}
\end{theorem}

Heuristically,~\eqref{eq:int-1} states that the Kolmogorov distance between the laws of $X_t/\sqrt{t}$ and $\sigma_t Z$ decays faster than say $1/\log(t)$ as $t\to\infty$. The parameter $\kappa\ge 1$ is chosen to ensure $\sigma_t>0$, with its precise value not being important for~\eqref{eq:int-1}. We stress that $\sigma_t$ in~\eqref{eq:int-1} \emph{cannot}, in general, be replaced by $\sigma$. The next result characterises (in terms of the L\'evy measure $\nu$ of $X$) whether such a substitution is valid. 

\begin{theorem}
\label{thm:int-2}
With $X$ as in Theorem~\ref{thm:int-1}, 
the following are equivalent:
\begin{equation}
\label{eq:int-2}
\int_1^\infty 
\sup_{x\in\R}\big|
    \p\big(X_t/\sqrt{t}\le x\big)
    -\Phi\big(x/\sigma\big)\big|\frac{\D t}{t}<\infty
\end{equation}
and 
$\E[X_1^2\max\{0,\log(X_1)\}]<\infty$.
\end{theorem}

Differently put, Theorem~\ref{thm:int-2} states that the Kolmogorov distance between the laws of $X_t/\sqrt{t}$ and $\sigma Z$
satisfies the integrability condition in~\eqref{eq:int-2}
if and only if 
the
integral    
$\int_{\R\setminus(-1,1)} x^2\log(x)\nu(\D x)$
is finite.
 In particular, Theorem~\ref{thm:int-2} characterises explicitly the L\'evy processes for which~\eqref{eq:int-2} does not hold but~\eqref{eq:int-1} does. The main idea behind the proofs of Theorems~\ref{thm:int-1} and~\ref{thm:int-2} is to 
apply Berry-Esseen bounds to a L\'evy process possessing all moments, obtained by removing from a path of  
$X$ the finitely many jumps with magnitude greater than $\kappa\sqrt{t}$ during the time interval 
$[0,t]$.

\section{Proofs}
\label{sec:proofs}

Let $(\Sigma^2,\nu,\beta)$ be the generating triplet of the L\'evy process $X=(X_t)_{t\ge 0}$ corresponding to the cutoff function $x\mapsto\1_{\{|x|<1\}}$,
where $\Sigma^2\geq0$ is the variance of the Gaussian component of $X$ and $\beta$ is a parameter
in $\R$
(see~\cite[Def.~8.2]{SatoBookLevy} for details). All generating triplets in this paper are with respect to the cutoff function $x\mapsto\1_{\{|x|<1\}}$. We refer to  monograph~\cite{SatoBookLevy} for background on L\'evy processes.

\begin{proof}[Proof of Theorem~\ref{thm:int-1}]
For any $t\ge 1$, let
$\tilde Y^{(t)}=(\tilde Y^{(t)}_s)_{s\geq0}$
be the compound Poisson process consisting of jumps of $X$
with magnitude at least $\kappa \sqrt{t}$
and define 
$Y^{(t)}=(Y^{(t)}_s)_{s\geq0}$ 
as 
$Y^{(t)}_s:=X_s-\tilde Y^{(t)}_s$.
Then, by~\cite[Thm~19.2]{SatoBookLevy},
$Y^{(t)}$ is
a L\'evy process with generating triplet  $(\Sigma^2,\nu|_{(-\kappa\sqrt{t},\kappa\sqrt{t})},\beta)$
 whose jumps are thus of magnitude smaller than $\kappa \sqrt{t}$.
 Since the support of the L\'evy measure of $Y^{(t)}$ is compact, by~\cite[Thm~25.3]{SatoBookLevy}, $Y^{(t)}_t$ has moments of all orders. Thus we may define the real number $\mu_t\coloneqq \E Y_t^{(t)}$. Moreover, the constant $\kappa\geq1$, chosen in the statement of Theorem~\ref{thm:int-1}, ensures $$0<\sigma^2_t=\Sigma^2+\int_{(-\kappa\sqrt{t},\kappa\sqrt{t})}x^2\nu(\D x)=\Var(Y_t^{(t)})/t<\infty\quad\text{ for all $t\geq1$.}$$
 The first equality in the last display follows from the identity
 $\sigma^2=\Sigma^2+\int_{\R}x^2\nu(\D x)$, which holds by~\cite[Example~25.12]{SatoBookLevy} applied to $X$. The same argument applied to the L\'evy process $Y^{(t)}$ yields the second equality in the display. 

 Define the function 
$$K(t)\coloneqq\sup_{x\in\R}
    |\p(X_t/\sqrt{t}\le x)
        -\Phi(x/\sigma_t)|\quad\text{ for all $t>0$.}$$
Let $J_{t}$ denote the event on which $X$ only has jumps of magnitude smaller than $\kappa\sqrt{t}$ during the time interval $[0,t]$. Note that on the event $J_t$ we have $X_t=Y_t^{(t)}$, implying the inequality
$$|\p(X_t\le x)-\p(Y_t^{(t)}\le x)|
\le\E |\1_{\{X_t\leq x\}}-\1_{\{Y_t^{(t)}\leq x\}}| 
\le\E[\1_{J_t^\co}]=\p(J_t^\co)
\quad\text{for all $x\in\R$ and $t\ge 1$.}$$ 
By adding and subtracting the probability $\p(Y_t^{(t)}/\sqrt{t}\le x)$
in the definition of $K(t)$, for all $t\geq1$ we obtain the inequality 
\begin{equation}
\label{eq:A_less_than_B_plus}
    K(t)\le A(t)+\p(J_{t}^\co),\quad \text{
where}\quad  
A(t)\coloneqq\sup_{x\in\R}\big|
    \p\big(Y^{(t)}_t/\sqrt{t}\le x\big)
    -\Phi\big(x/\sigma_t\big)\big|.
    \end{equation} 
The triangle inequality implies 
\[
A(t)
=\sup_{x\in\R}\big|
    \p\big((Y^{(t)}_t-\mu_t)/\sqrt{t}\le x\big)
    -\Phi\big((x+\mu_t/\sqrt{t})/\sigma_t\big)\big|
\leq B(t)+C(t),
\]
where for any $t\geq1$ we define 
\begin{equation*}
B(t)\coloneqq\sup_{x\in\R} \big|\p\big(
    (Y^{(t)}_t-\mu_t)/\sqrt{t}\le x\big)
        -\Phi\big(x/\sigma_t\big)\big|
\quad\&\quad
C(t)\coloneqq \sup_{x\in\R}
    \big|\Phi\big(x/\sigma_t\big)
        -\Phi\big((x+\mu_t/\sqrt{t})/\sigma_t\big)\big|.
\end{equation*}
To complete the proof, it suffices to show that the following integrals are finite:
\begin{equation*}
\text{\nf (a)}\enskip
\int_1^\infty \p(J_t^\co) \frac{\D t}{t}<\infty,
\qquad\text{\nf (b)}\enskip
\int_1^\infty B(t) \frac{\D t}{t}<\infty,
\qquad\text{\nf (c)}\enskip
\int_1^\infty C(t) \frac{\D t}{t}<\infty.
\end{equation*}

The integrals in (a)--(c) exist since the integrands are non-negative. It remains to prove they are finite. Fubini's theorem yields  
\begin{equation}
\label{eq:I_expression}
I\coloneqq \int_\R x^2\nu(\D x)
=\int_0^\infty 2x\ov\nu(x)\D x
=\int_0^\infty \ov\nu(\sqrt{x})\D x<\infty, \end{equation}
where $\nu$ is the L\'evy measure of $X$ and $\ov\nu(x)\coloneqq \nu(\R\setminus(-x,x))$, $x>0$. 

(a) Since $\tilde Y^{(t)}=X-Y^{(t)}$ is a compound Poisson process with intensity $\ov\nu(\kappa\sqrt{t})$, 
the first jump of $\tilde Y^{(t)}$ is exponentially distributed with mean $1/\ov\nu(\kappa\sqrt{t})$.
As the event $J_t$ can be defined by the first jump of $\tilde Y^{(t)}$ being greater than $t$,  it has probability $\p(J_t)=e^{-t\ov\nu(\kappa\sqrt{t})}$. Thus, we have $$\p(J_{t}^\co)=1-e^{-t\ov\nu(\kappa\sqrt{t})}\le t\ov\nu(\kappa\sqrt{t}), \quad\text{for $t>0$},$$
implying the bound 
$\int_1^\infty t^{-1}\p(J_t^\co) \D t\le \int_1^\infty \ov\nu(\kappa\sqrt{t})\D t\leq I/\kappa^2$.

(b) For any $t\ge1$, $Y^{(t)}_t$ is nontrivial and infinitely divisible with a finite third moment. More precisely, $Y^{(t)}_t =\sum_{k=1}^n Z_k$, where the variables
$Z_k\coloneqq Y^{(t)}_{tk/n}-Y^{(t)}_{t(k-1)/n}\eqd Y^{(t)}_{t/n}$ are independent. 
The Berry-Esseen inequality for independent random variables yields a constant $c>0$ such that 
\begin{align*}
B(t)
&=\sup_{y\in\R}\bigg|\p\bigg( 
    \frac{\sum_{k=1}^n Z_k-\E\sum_{k=1}^n Z_k}
        {\Var{\big( \sum_{k=1}^n Z_k\big)}^{1/2}}\le y\bigg)
    - \Phi(y)\bigg|
\le 
cn \E \big[|Y^{(t)}_{t/n}
    -\E Y^{(t)}_{t/n}|^3\big]/
        \big( n \Var(Y^{(t)}_{t/n})\big)^{3/2}\\
 &\le 4cn \big(\E \big[\big|Y^{(t)}_{t/n}\big|^3\big] 
    + \big|\E \big[Y^{(t)}_{t/n}\big]\big|^3\big)/\big(n\Var\big(Y^{(t)}_{t/n}\big)\big)^{3/2}\qquad
    \text{for all $n\in\N$.}
\end{align*}
The second inequality in the display above follows from the inequality $|(a+b)/2|^p\le (|a|^p+|b|^p)/2$ for any $a,b\in\R$ and $p\ge 1$ (which holds by convexity), applied with $a=Y_{t/n}^{(t)}$, $b=-\E Y_{t/n}^{(t)}$ and $p=3$. 
Since, by~\cite[Thm~1.1]{SmalltimemomentasympLopez}, the limit $\lim_{n\to\infty}n
    \E\big[\big|Y^{(t)}_{t/n}\big|^3\big]
=t\int_{(-\kappa\sqrt{t},\kappa\sqrt{t})}|x|^3\nu(\D x)$ holds, the equalities $\E\big[Y^{(t)}_{t/n}\big]=\E\big[Y^{(t)}_{1}\big]t/n$ and $\Var\big(Y^{(t)}_{t/n}\big)=\Var\big(Y^{(t)}_1\big)t/n$ imply
\begin{align*}
B(t)
&\le \lim_{n\to\infty}
\frac{4cn(\E\big[|Y^{(t)}_{t/n}|^3\big]+\big|\E\big[Y_{t/n}^{(t)}\big]\big|^3)}
    {\big(n\Var\big(Y^{(t)}_{t/n}\big)\big)^{3/2}}
=\frac{4c(\lim_{n\to\infty}n\E\big[|Y^{(t)}_{t/n}|^3\big] 
+\lim_{n\to\infty}\big|\E\big[Y_1^{(t)}\big]\big|^3t^3/n^2)}
    {\big(\Var\big(Y^{(t)}_1\big)t\big)^{3/2}}\\
&=\frac{4ct\int_{(-\kappa\sqrt{t},\kappa\sqrt{t})} 
    |x|^3 \nu(\D x)}
    {\big( \Sigma^2 t 
    + t\int_{(-\kappa\sqrt{t},\kappa\sqrt{t})} x^2\nu(\D x)\big)^{3/2}}\le 
    \frac{4c }{\sqrt{t}\sigma_1^{3}}
    \int_{(-\kappa\sqrt{t},\kappa\sqrt{t})}|x|^3\nu(\D x)
\end{align*}
for any $t\geq1$ (recall that $t\mapsto \sigma_t^2$, defined in Theorem~\ref{thm:int-1}, is non-decreasing).
We thus obtain
\begin{equation}
\label{eq:bound_on_B(t)_integral}
\int_1^\infty B(t)\frac{\D t}{t}\le
  \frac{4c}{\sigma_1^3} \int_1^\infty 
t^{-3/2}\int_{(-\kappa\sqrt{t},\kappa\sqrt{t})}|x|^3\nu(\D x)\D t\leq \frac{12c}{\sigma_1^3}\int_1^\infty 
t^{-3/2}\int_0^{\kappa \sqrt{t}} x^2\ov \nu(x)\D x\D t,
\end{equation}
where the second inequality follows from
the identity $\int_{(-w,w)}|x|^3\nu(\D x)
=-w^3\ov\nu(w)+3\int_0^w x^2\ov\nu(x)\D x$ for all $w>0$. 
The limit $0\leq y^2\ov\nu(y)
\leq \int_{\R\setminus(-y,y)}x^2\nu(\D x)\to0$ as $y\to\infty$ 
implies 
$\int_0^{\kappa\sqrt{T}}x^2\ov\nu(x)\D x/\sqrt{T}\to0$
as $T\to\infty$.
Thus, the bound in~\eqref{eq:bound_on_B(t)_integral}
and integration-by-parts
imply that the integral in (b) is finite:
\begin{align*}
\int_1^\infty t^{-3/2}
    \int_0^{\kappa\sqrt{t}} x^2\ov\nu(x)\D x\D t
&=  \bigg[-2t^{-1/2}\int_0^{\kappa\sqrt{t}} x^2\ov\nu(x)\D x\bigg]_{1}^\infty
    + \int_1^\infty 2t^{-1/2}\cdot \kappa^2t\ov\nu(\kappa\sqrt{t})\cdot \frac{\kappa}{2\sqrt{t}}\D t\\
&=  2\int_0^\kappa x^2\ov\nu(x)\D x
    + \kappa^3\int_1^\infty\ov\nu(\kappa\sqrt{t})\D t
= 2\int_0^\kappa x^2\ov \nu(x)\D x + \kappa \int_{\kappa^2}^\infty \ov \nu(\sqrt{y}) \D y \\
&
\le 2\int_0^\kappa x^2\ov \nu(x)\D x + \kappa I,
\end{align*}
where the final inequality follows from~\eqref{eq:I_expression}.

(c) Since the distribution $\Phi$ is unimodal and symmetric, the mean-value theorem implies that $C(t)$ satisfies 
\begin{align*}
C(t)
    =\big|\Phi\big(
            \mu_t/(2\sqrt{t}\sigma_t)\big)
        -\Phi(-\mu_t/(2\sqrt{t}\sigma_t)\big|
=e^{-c^2/2}|\mu_t|/(\sqrt{2\pi t}\sigma_t) \leq
|\mu_t|/(\sigma_1\sqrt{t})\quad \text{for $t\geq1$} 
\end{align*} 
and some
$c\in(-|\mu_t|/(2\sqrt{t}\sigma_t),|\mu_t|/(2\sqrt{t}\sigma_t))$ (recall that $t\mapsto \sigma_t^2$ is non-decreasing). Since $\sigma_1>0$, it suffices to prove that $\int_1^\infty |\mu_t|t^{-3/2}\D t<\infty$.
By $0=\E[X_t]=\beta t+t\int_{\R\setminus(-1,1)}x\nu(\D x)$, we have  $\mu_t=\E[Y_t^{(t)}]=-t\int_{\R\setminus(-\kappa\sqrt{t},\kappa\sqrt{t})}x\nu(\D x)$. Hence $|\mu_{t}|\le t\int_{\R\setminus(-\sqrt{t},\sqrt{t})}|x|\nu(\D x)$ for all $t\ge 1$ since  $\kappa\geq 1$.  Apply Fubini’s theorem to obtain  
\begin{equation*}
\int_1^\infty \frac{|\mu_t|}{t^{3/2}}\D t
\le \int_1^\infty \frac{1}{\sqrt{t}} \int_{\R\setminus(-\sqrt{t},\sqrt{t})}|x|\nu(\D x)\D t
\leq 2\int_{\R\setminus(-1,1)}x^2
\nu(\D x)
<\infty.
\end{equation*}
This implies that the integral in (c) is finite, completing the proof of Theorem~\ref{thm:int-1}.
\end{proof}

\begin{proof}[Proof of Theorem~\ref{thm:int-2}]
Define the function $\varphi(a)
\coloneqq\sup_{x\in\R}|\Phi(x)-\Phi(ax)|$, for $a\in(0,1)$. 
Note that the function $x\mapsto|\Phi(x)-\Phi(ax)|$ is symmetric around $0$ and its maximal value on the positive half-line is attained when its derivative is $0$. Elementary calculations reveal that this critical point on the positive half-line equals $z(a):=\sqrt{2\log(1/a)/(1-a^2)}$, implying $\varphi(a)=\Phi(z(a))-\Phi(az(a))$. 

Let $K(t,x)\coloneqq|\p(X_t/\sqrt{t}\le x)
-\Phi(x/\sigma_t)|$ for $t>0$ and $x\in\R$ and recall  $K(t)=\sup_{x\in\R}K(t,x)$. Since $\varphi(\sigma_t/\sigma)=\sup_{x\in\R}|\Phi(x/\sigma_t)-\Phi(x/\sigma)|$, the triangle inequality yields  
\[
K(t) + \varphi(\sigma_t/\sigma)
\ge\sup_{x\in\R}\left|\p\left(\frac{X_t}{\sqrt{t}}\le x\right)
    -\Phi\left(\frac{x}{\sigma}\right)\right|
\ge\varphi(\sigma_t/\sigma)-K(t)\qquad\text{for all $t\geq1$.}
\]
By Theorem~\ref{thm:int-1} we have $\int_1^\infty t^{-1}K(t)\D t<\infty$, so the integral in~\eqref{eq:int-2} is finite if and only if 
\begin{equation}
\label{eq:int-3}
\int_1^\infty \varphi(\sigma_t/\sigma)\frac{\D t}{t}<\infty.
\end{equation}
Thus, it suffices to show that~\eqref{eq:int-3} is equivalent to $\int_{\R\setminus(-1,1)} x^2\log(x)\nu(\D x)<\infty$.

Note that, as $a\uparrow1$, we have 
$\log(1/a)/(1-a)\to 1$
and $z(a)\to 1$. The mean value theorem implies 
\[
\frac{\varphi(a)}{1-a}
=\frac{\Phi(z(a))-\Phi(az(a))}{1-a}
=\frac{1}{1-a}\int_{az(a)}^{z(a)}
    \frac{1}{\sqrt{2\pi}}e^{-x^2/2}\D x
=\frac{z(a)}{\sqrt{2\pi}}
    e^{-c(a)^2/2}
\to\frac{1}{\sqrt{2e\pi}}
\]
for some $c(a)\in(az(a),z(a))$. Since $\sigma_t/\sigma\uparrow 1$ as $t\to\infty$, for all sufficiently large times $t$ we have 
$(\sigma-\sigma_t)/\sqrt{4e\pi\sigma^2} \leq \varphi(\sigma_t/\sigma) \leq (\sigma-\sigma_t)/\sqrt{e\pi\sigma^2}$, implying that~\eqref{eq:int-3} is equivalent to $\int_1^\infty t^{-1}(\sigma - \sigma_t)\D t<\infty$. Moreover, Fubini's theorem gives 
\begin{align*}
\int_1^\infty (\sigma - \sigma_t)\frac{\D t}{t}
&=\int_1^\infty \int_{\R\setminus(-\kappa\sqrt{t},\kappa\sqrt{t})}x^2\nu(\D x)\frac{\D t}{t}\\
&=\int_{\R\setminus(-\kappa,\kappa)}x^2 \int_1^{x^2/\kappa^2}\frac{\D t}{t}\nu(\D x)
=\int_{\R\setminus(-\kappa,\kappa)}
    x^2\log(x^2/\kappa^2)\nu(\D x).
\end{align*}
The last integral is finite if and only if $\int_{\R\setminus(-1,1)} x^2\log(x)\nu(\D x)<\infty$, so the result follows.
\end{proof}

\printbibliography

\section*{Acknowledgements}

\thanks{
\noindent AM was supported by EPSRC grant EP/P003818/1 and the Turing Fellowship funded by the Programme on Data-Centric Engineering of Lloyd's Register Foundation;
JGC and AM are supported by The Alan Turing Institute under the EPSRC grant EP/N510129/1; 
JGC is supported by CoNaCyT scholarship 2018-000009-01EXTF-00624 CVU 699336; DB is funded by CDT in Mathematics and Statistics at The University of Warwick.

\vspace{1mm}

\noindent We would like to thank Nicholas H. Bingham for the encouragement to write up this note and an anonymous Referee and Associate Editor
for comments that improved the presentation in the paper.}

\end{document}